\documentclass[preprint,12pt]{elsarticle}

\usepackage{amsthm}
\usepackage{amssymb,amsmath}
\usepackage{pstricks}
\usepackage{ifthen}
\usepackage{calc}

\theoremstyle{plain}
\newtheorem{thm}{Theorem}
\newtheorem{lem}[thm]{Lemma}
\newtheorem{prop}[thm]{Proposition}
\newtheorem{cor}[thm]{Corollary}

\theoremstyle{definition}

\newtheorem{prob}{Problem}
\newtheorem{conj}{Conjecture}
\newtheorem{rem}{Remark}
\newtheorem{fact}{Fact}

\newproof{pf}{Proof}

\newproof{pot}{Proof of Theorem \ref{main theorem}}

 \makeatletter
\@addtoreset{thm}{chapter}
\@addtoreset{thm}{section}
\@addtoreset{thm}{subsection}

\makeatother

\begin{document}

\begin{frontmatter}

\title{A transformation that preserves  principal minors of skew-symmetric matrices }

\author[B]{Abderrahim  Boussa\"{\i}ri\corref{cor1}}
\ead{aboussairi@hotmail.com}

\author[B]{Brahim Chergui}
\ead{cherguibrahim@gmail.com}

\cortext[cor1]{Corresponding author}

\address[B]{Facult\'e des Sciences A\"{\i}n Chock, D\'epartement de
Math\'ematiques et Informatique,  Laboratoire de Topologie, Alg\`{e}bre, G\'{e}om\'{e}trie et Math\'{e}matiques discr\`{e}tes

Km 8 route d'El Jadida,
BP 5366 Maarif, Casablanca, Maroc}

\begin{abstract}
 Our motivation comes from the work of Engel and Schneider (1980). Their main
theorem implies that two symmetric matrices have equal corresponding principal
minors of all orders if and only if they are diagonally similar. This study
was continued by Hartfiel and Loewy (1984). They found sufficient conditions
under which two $n\times n$ matrices\ $A$ and $B$ have equal corresponding
principal minors of all orders if and only if  $B$ or its transpose
$B^{t}$ is diagonally similar to $A$. In this paper, we give a new way to
construct a pair of skew-symmetric having equal corresponding principal minors
of all orders.

\end{abstract}

\begin{keyword}

Skew-symmetric matrix; Principal minor; Diagonal similarity; Graph; digraph; Orientation.

\MSC 15A15, 05C50
\end{keyword}

\end{frontmatter}


\section{Introduction}

Throughout this paper, all matrices are real or complex. The identity matrix
of order $n$ is denoted by $I_{n}$ and the transpose of a matrix $A$ by
$A^{t}$. A \emph{minor} of a matrix $A$ is the determinant of a square
submatrix of $A$, and the determinant of a principal submatrix is a\emph{
principal minor}. The \emph{order} of a minor is $k$ if it is the determinant
of a $k\times k$ submatrix.

In this work, we consider the following Problem.

\begin{prob}
\label{problem principal}What is the relationship between two matrices having
equal corresponding principal minors of all orders ?
\end{prob}

For symmetric matrices, this Problem has been solved by Engel and Schneider
\cite{Engel}. More precisely, it follows from their work (see Theorem 3.5)
that two symmetric matrices $A$, $B$ have equal corresponding principal minors
of all orders if and only if there exists a  $\{-1,1\}$ diagonal matrix $D$  such that $B=D^{-1}AD$.

Consider now two arbitrary $n\times n$ matrices $A$ and $B$. We say that $A$,
$B$ are \emph{diagonally similar up to transposition} if there exists a
nonsingular diagonal matrix $D$ such that $B=D^{-1}AD$ or $B^{t}=D^{-1}AD$.
Clearly, diagonal similarity up to tansposition preserves all principal
minors. But, as observed in \cite{Engel} and \cite{HL} (see Remark
\ref{remaque base} below), this is not, in general, the unique way to
construct a pair of matrices having equal principal minors.

\begin{rem}
\label{remaque base} Consider the following skew-symmetric matrices :
\[
A:=\left(
\begin{array}
[c]{cc}%
A_{11} & A_{12}\\
-A_{12}^{t} & A_{22}%
\end{array}
\right)  \text{ and }B:=\left(
\begin{array}
[c]{cc}%
-A_{11} & A_{12}\\
-A_{12}^{t} & A_{22}%
\end{array}
\right)
\]

where $A_{11}$, $A_{22}$ are square matrices.

We will see in Proposition \ref{inversion and determinant} that if rank
$A_{12}\leq1$, then $A$ and $B$ have equal corresponding principal minors of
all orders. However, these matrices are not always diagonally similar up to transposition.
\end{rem}

Hartfiel and Loewy \cite{HL}, and then Loewy \cite{Lw86} considered a class of
matrices excluding the situation of the previous Remark. Their work concerns
irreducible matrices with an additional condition. In order to state the main
theorem of Loewy \cite{Lw86}, we need the following definitions and notations.
Let $A=[a_{ij}]$ be an $n\times n$ matrix and let $X,Y$ \ be two nonempty
subsets of $\left[  n\right]  $ (where $\left[  n\right]  :=\left \{
1,\ldots,n\right \}  $). We denote by $A[X,Y]$ the submatrix of $A$ having row
indices in $X$ and column indices in $Y$. If $X=Y$, then $A[X,X]$ is a
principal submatrix of $A$ and we abbreviate this to $A[X]$. A square matrix
$A$ is \emph{irreducible }if there exists no permutation matrix $P$, so that
$A$ can be reduced to the form $PAP^{T}=\left(
\begin{array}
[c]{cc}%
X & Z\\
0 & Y
\end{array}
\right)  $ where $X$ and $Y$ are square matrices.

The main theorem of Loewy \cite{Lw86} is stated as follows.

\begin{thm}
\label{loewy copy} Let $A,$ $B$ be two $n\times n$ matrices. Suppose $n\geq4$,
$A$ irreducible and for every partition of $\left[  n\right]  $ into two
subsets $X,Y$\ with $\left \vert X\right \vert \geq2$, $\left \vert Y\right \vert
\geq2$, either rank $A[X,Y]\geq2$ or rank $A[Y,X]\geq2$. If $A$ and $B$ have
equal corresponding principal minors of all orders, then they are diagonally
similar up to transposition.
\end{thm}

For skew-symmetric matrices with no zeros off the diagonal, we have improved
this theorem in \cite{BOCH} by considering only the principal minors of order
at most $4$.

We will describe now another way to construct a pair of skew-symmetric
matrices having equal corresponding principal minors of all orders. Let
$A=\left[  a_{ij}\right]  $ be a $n\times n$ matrix. Following \cite{BOCH}, a
subset $X$ of $\left[  n\right]  $ is a \emph{HL-clan} of $A$ if both of
matrices $A\left[  X,\overline{X}\right]  $ and $A\left[  \overline
{X},X\right]  $ have rank at most $1$ (where $\overline{X}:=\left[  n\right]
\setminus X$). By definition, $\emptyset$, $\left[  n\right]  $ and singletons
are HL-clans. Consider now the particular case when $A$ is skew-symmetric and
let $X$ be a subset of $\left[  n\right]  $. We denote by $Inv(X,A):=[t_{ij}]$
the matrix obtained from $A$ as follows. For any $i,j\in \left[  n\right]  $,
$t_{ij}=-a_{ij}$ if $i,j\in X$ and $t_{ij}=a_{ij}$, otherwise. As we have
mentioned in Remark \ref{remaque base}, if $X$ is an HL-clan of $A$, then
$Inv(X,A)$ and $A$ have equal corresponding principal minors of all orders.
More generally, let $A$ and $B$ two skew-symmetric matrices and assume that
there exists a sequence $A_{0}=A,\ldots,A_{m}=B$ of $n\times n$ skew-symmetric
matrices such that for $k=0,\ldots,m-1$, $A_{k+1}=Inv(X_{k},A_{k})$ where
$X_{k}$ is a HL-clan of $A_{k}$. It easy to see that $A$ and $B$ have equal
corresponding principal minors. Two matrices $A,B$ obtained in this way are
called \emph{HL-clan-reversal-equivalent}. This defines an equivalence
relation between $n\times n$ skew-symmetric matrices which preserves principal
minors. In the converse direction, we propose the following conjecture.

\begin{conj}
\label{conjecture}Two $n\times n$ skew-symmetric real matrices have equal
corresponding principal minors of all order if and only if they are HL-clan-reversal-equivalent.
\end{conj}

We will restrict ourselves to the class $\mathcal{M}_{n}$ of $n\times n$
skew-symmetric matrices with entries from $\{-1,0,1\}$ and such that all
off-diagonal entries of the first row are nonzero. We obtain the following
Theorem, which is a partial answer to the conjecture above.

\begin{thm}
\label{main theorem} Let $A,B\in \mathcal{M}_{n}$. Then, the following
statements are equivalent:

\begin{description}
\item[i)] $A$ and $B$ have equal corresponding principal minors of order $4$;

\item[ii)] $A$ and $B$ have equal corresponding principal minors of all orders;

\item[iii)] $A$ and $B$ are HL-clan-reversal-equivalent.
\end{description}
\end{thm}

\section{HL-clan-reversal-equivalence}

In this section, we present some properties of HL-clan-reversal-equivalence.
We start with the following basic facts. Let $A=[a_{ij}]$ be a skew-symmetric
$n\times n$ matrix.

\begin{fact}
\label{f1}If $D=[d_{ij}]$ is a nonsingular diagonal matrix then $A$ and
$D^{-1}AD$ have the same HL-clans.
\end{fact}

\begin{proof}
Let $X$ be a subset of $[n]$. We have the following equalities:
\begin{align*}
(D^{-1}AD)\left[  \overline{X},X\right]   &  =(D^{-1}\left[  \overline
{X}\right]  )(A\left[  \overline{X},X\right]  )(D\left[  X\right]  )\\
(D^{-1}AD)\left[  X,\overline{X}\right]   &  =(D^{-1}\left[  X\right]
)(A\left[  X,\overline{X}\right]  )(D\left[  \overline{X}\right]  )
\end{align*}
But, the matrices $D\left[  X\right]  $ and $D\left[  \overline{X}\right]  $
are nonsingular, then $(D^{-1}AD)\left[  \overline{X},X\right]  $ and
$A\left[  \overline{X},X\right]  $ (resp. $(D^{-1}AD)\left[  X,\overline
{X}\right]  $ and $(A\left[  X,\overline{X}\right]  )$ have the same rank.
Therfore, $A$ and $D^{-1}AD$ have the same HL-clans.
\end{proof}

\begin{fact}
\label{f2}If $C$ be an HL-clan of $A$ then it is an HL-clan of $Inv(C,A)$.
\end{fact}

It suffices to see that
\begin{align*}
A\left[  C,\overline{C}\right]   &  =Inv(C,A)\left[  C,\overline{C}\right] \\
A\left[  \overline{C},C\right]   &  =Inv(C,A)\left[  \overline{C},C\right]
\end{align*}

\begin{fact}
\label{f3}If $C$ be an HL-clan of $A$ and $X$ is a subset of $\left[
n\right]  $, then $C\cap X$ is an HL-clan of $A[X]$ and $Inv(C,A)[X]=Inv(C\cap
X,A[X])$.
\end{fact}

\begin{proof}
We have rank $(A\left[  C\cap X,X\setminus(C\cap X)\right]  )\leq$rank
$(A\left[  C,\overline{C}\right]  )\leq1$ because $A\left[  C\cap
X,X\setminus(C\cap X)\right]  $ is a submatrix of $A\left[  C,\overline
{C}\right]  $ and $C$ is an HL-clan of $A$. Analougsly, we have rank
$(A\left[  X\setminus(C\cap X),C\cap X\right]  )\leq$rank $(A\left[
\overline{C},C\right]  )\leq1$. It follows that $C\cap X$ is an HL-clan of
$A[X]$. The second statement is trivial.
\end{proof}

The next Proposition states that HL-clan-reversal-equivalence generalizes
diagonal similarity up to transposition.

\begin{prop}
\label{DSTimpliHLRE} Let $A=[a_{ij}]$ and $B=[b_{ij}]$ be two $n\times n$
skew-symmetric matrices. If $A$ and $B$ are diagonally similar up to
transposition then they are HL-clan-reversal-equivalent.
\end{prop}

\begin{proof}
Let $A=[a_{ij}]$ and $B=[b_{ij}]$ be two $n\times n$ skew-symmetric matrices
diagonally similar up to transposition. As $B^{t}=-B=Inv([n],B)$, we can
assume that $B=\Delta^{-1}A\Delta$ for some nonsingular diagonal matrix
$\Delta$. It is easy to see that $b_{ij}=\pm a_{ij}$ for $i,j\in \lbrack n]$
and hence $\Delta$ may be chosen to be a $\left \{  -1,1\right \}  $-diagonal
matrix. We conclude by Lemma \ref{general inversion} below.
\end{proof}

\begin{lem}
\label{general inversion} Let $A=[a_{ij}]$ be an $n\times n$ skew-symmetric
matrix and let $D$ be a $\left \{  -1,1\right \}  $-diagonal matrix. Then $A$
and $D^{-1}AD$ are HL-clan-reversal-equivalent.
\end{lem}

\begin{proof}
We denote by $d_{1},d_{2},\ldots,d_{n}$ the diagonal entries of $D$. Let
$U_{D}:=\left \{  i\in \lbrack n]:d_{i}=-1\right \}  $. We will show by induction
on $t:=\left \vert U_{D}\right \vert $ that there exists a sequence
$A_{0}=A,\ldots,A_{m}=D^{-1}AD$ of $n\times n$ skew-symmetric matrices such
that for $k=0,\ldots,m-1$, $A_{k+1}=Inv(X_{k},A_{k})$ where $X_{k}=\emptyset$,
$X_{k}=[n]$ or $[n]\setminus X_{k}$ is a singleton. If $t=0$ then $D^{-1}AD=A$
and hence it suffices to take $m=1$, $A_{0}=A$ and $X_{0}=\emptyset$. Now
assume that $t>0$. Let $j\in U_{D}$ and consider the diagonal matrix
$\Delta^{(j)}=diag(\delta_{1},\ldots,\delta_{n})$ where $\delta_{j}=-1$ and
$\delta_{i}=1$ if $i\neq j$. Clearly $n_{D\Delta^{(j)}}=t-1$ and then, by
induction hypothesis, there exists a sequence $A_{0}$ $=A,\ldots
,A_{m}=(D\Delta^{(j)})^{-1}AD\Delta^{(j)}$ of $n\times n$ skew-symmetric
matrices such that for $k=0,\ldots,m-1$, $A_{k+1}=Inv(X_{k},A_{k})$ where
$X_{k}=\emptyset$, $X_{k}=[n]$ or $[n]\setminus X_{k}$ is a singleton. To
prove that $A$ and $D^{-1}AD$ are HL-clan-reversal-equivalent, it suffices to
extend the sequence $A_{0}$ $=A,\ldots,A_{m}$ by adding two terms,
$A_{m+1}:=Inv([n],A_{m})$ and $A_{m+2}:=Inv([n]\setminus \left \{  j\right \}
,A_{m+1})$.
\end{proof}

The following Proposition appears in another form in \cite{HL} (see Lemma 5).

\begin{prop}
\label{inversion and determinant} Let $A=[a_{ij}]$ be a skew-symmetric
$n\times n$ matrix. If $X$ is an HL-clan of $A$ then $\det(Inv(X,A))=\det(A)$.
\end{prop}

\begin{proof}
Without loss of generality, we can assume that $X=\left \{  1,\ldots,p\right \}
$. We will show that $A$ and $Inv(X,A)$ have the same characteristic
polynomial. As $X$ is an HL-clan of $A$, the submatrix $A[\overline{X},X]$ has
rank at most $1$ and hence there are two column vectors $\alpha=\left(
\begin{array}
[c]{c}%
\alpha_{p+1}\\
\vdots \\
\alpha_{n}%
\end{array}
\right)  $ and $\beta=\left(
\begin{array}
[c]{c}%
\beta_{1}\\
\vdots \\
\beta_{p}%
\end{array}
\right)  $ such that $A[\overline{X},X]=\alpha \beta^{t}$.

Let $A[X]:=A_{11}$ and $A[\overline{X}]:=A_{22}$. Then $A=\left(
\begin{array}
[c]{cc}%
A_{11} & -\beta \alpha^{t}\\
\alpha \beta^{t} & A_{22}%
\end{array}
\right)  $ and $Inv(X,A)=\left(
\begin{array}
[c]{cc}%
A_{11}^{t} & -\beta \alpha^{t}\\
\alpha \beta^{t} & A_{22}%
\end{array}
\right)  $, where $A_{11}^{t}=$ $-A_{11}$. We will prove that $A$ and
$Inv(X,A)$ have the same characteristic polynomial.

Let $\lambda$ satisfying $\left \vert \lambda \right \vert >\lambda_{0}$ where
$\lambda_{0}$ is the spectral radius of $A_{11}$. Then $A_{11}+\lambda I_{p}$
is nonsingular and hence, by using the Schur complement, we have

$%
\begin{array}
[c]{ccl}%
\det(A+\lambda I_{n}) & = & \det(A_{11}+\lambda I_{p})\det(A_{22}+\lambda
I_{n-p}+\alpha \beta^{t}(A_{11}+\lambda I_{p})^{-1}\beta \alpha^{t})\\
& = & \det(A_{11}+\lambda I_{p})\det(A_{22}+\lambda I_{n-p}+(\beta^{t}%
(A_{11}+\lambda I_{p})^{-1}\beta)\alpha \alpha^{t})\\
& = & \det((A_{11}+\lambda I_{p})^{t})\det(A_{22}+\lambda I_{n-p}+(\beta
^{t}(A_{11}+\lambda I_{p})^{-1}\beta)^{t}\alpha \alpha^{t})\\
& = & \det((A_{11}^{t}+\lambda I_{p}))\det(A_{22}+(\beta^{t}(A_{11}%
^{t}+\lambda I_{p})^{-1}\beta)\alpha \alpha^{t})\\
& = & \det(Inv(X,A)+\lambda I_{n})
\end{array}
$\  \  \  \  \  \  \  \  \  \  \  \  \  \

It follows that $A$ and $Inv(X,A)$ have the same characteristic polynomial and
then $\det(A)=\det(Inv(X,A))$.
\end{proof}

The following \ Collorary is a direct consequence of the previous Proposition
and Fact \ref{f3}.

\begin{cor}
\label{same pp}Let $A=[a_{ij}]$ be a skew-symmetric $n\times n$ matrix. If $X$
is an HL-clan of $A$ then $Inv(X,A)$ and $A$ have the same principal minors.
\end{cor}


\section{ Digraphs and orientation of a graph}
We start with some definitions about digraphs. A \emph{directed graph} or
\emph{digraph} $\Gamma$ consists of a nonempty finite set $V$ of
\emph{vertices} together with a (possibly empty) set $E$ of ordered pairs of
distinct vertices called \emph{arcs}. Such a digraph is denoted by $(V,E)$.
The \emph{converse }of a digraph $\Gamma$ denoted by $\Gamma^{\ast}$ is the
digraph obtained from $\Gamma$ by reversing the direction of all its arcs.

Let $\Gamma=(V,E)$ be a digraph and let $X$ be a subset of $V$. The
\emph{subdigraph} of $\Gamma$ \emph{induced} by $X$ is the digraph
$\Gamma \left[  X\right]  $ whose vertex set is $X$ and whose arc set consists
of all arc of $\Gamma$ which have end-vertices in $X$.

Two digraphs $\Gamma=(V,E)$ and $\Gamma^{\prime}=(V^{\prime},E^{\prime})$ are
said to be \emph{isomorphic} if there is a bijection $\varphi$ from $V$ onto
$V^{\prime}$ which preserves arcs, that is $(x,y)\in E$ if and only if
$(\varphi(x),\varphi(y))\in E^{\prime}$. Any such bijection is called an
\emph{isomorphism}. We say that $\Gamma$ and $\Gamma^{\prime}$ are
\emph{hemimorphic},\emph{ }if there exists an isomorphism from $\Gamma$ onto
$\Gamma^{\prime}$ or from $\Gamma^{\ast}$ onto $\Gamma^{\prime}$.

Let $\Gamma=(V,E)$ be a digraph. Following \cite{EHR ROZ}, a subset $X$ of $V$
is a \emph{clan} of $\Gamma$ if for any $a,b\in X$ and $x\in V\diagdown X$,
$(a,x)\in E$ (resp. $(x,a)\in E)$) if and only if $(b,x)\in E$ (resp.
$(x,b)\in E$). For a subset $X$ of $V$, we denote by $Inv(X,\Gamma)$ the
digraph obtained from $\Gamma$ by reversing all arcs of $\Gamma \left[
X\right]  $. Clearly, $Inv(X,Inv(X,\Gamma))=\Gamma$ and moreover, if $X$ is a
clan of $\Gamma$ then $X$ is a clan of $Inv(X,\Gamma)$.

Let $G=\left(  V,E\right)  $ be a simple graph (without loops and multiple
edge). An\emph{\ orientation} of $G$ is an assignment of a direction to each
edge of $G$ in order to obtain an directed graph $\overrightarrow{G}$. For
$x\neq y\in V$, $x\overset{\overrightarrow{G}}{\rightarrow}y$ means $(x,y)$ is
an arc of $\overrightarrow{G}$. For $Y\subseteq V$ and $x\in V\diagdown X$,
$x\overset{\overrightarrow{G}}{\rightarrow}Y$ means $x\overset{\overrightarrow
{G}}{\rightarrow}y$ for every $y\in Y$.

\begin{rem}
\label{fourcase}$\left.  {}\right.  $

\begin{description}
\item[i)] There are exactly four possible simple graphs with three vertices:
the complete graph $K_{3}$, the path $P_{2}$, the complement of these two
graphs, namely $\overline{K_{3}}$ and $\overline{P_{2}}$ (see Figure \ref{fig1});

\item[ii)] The path $P_{2}$ has two non-hemimorphic orientations $\Gamma_{1}$
and $\Gamma_{2}$ (see Figure \ref{fig2} (a));

\item[iii)] The complete graph $K_{3}$ has two non-hemimorphic orientations
$\Gamma_{3}$ and $\Gamma_{4}$ (see Figure \ref{fig2} (b)).
\end{description}
\end{rem}


\begin{figure}[h]
\begin{center}
\setlength{\unitlength}{0.7cm}
\begin{picture}(30,4)

\put(0,2){\circle*{.2}}
\put(1.5,4){\circle*{.2}}
\put(3,2){\circle*{.2}}
\put(1.3,1){$\overline{K_{3}}$}
\put(6,2){\circle*{.2}}
\put(7.5,4){\circle*{.2}}
\put(9,2){\circle*{.2}}
\thicklines
\put(6,2){\line(3,4){1.5}}
\put(7.3,1){$\overline{P_{2}}$}
\put(12,2){\circle*{.2}}
\put(13.5,4){\circle*{.2}}
\put(15,2){\circle*{.2}}
\thicklines
\put(12,2){\line(1,0){3}}
\put(12,2){\line(3,4){1.5}}
\put(13.1,1){$P_{2}$}
\put(18,2){\circle*{.2}}
\put(19.5,4){\circle*{.2}}
\put(21,2){\circle*{.2}}
\thicklines
\put(18,2){\line(1,0){3}}
\put(18,2){\line(3,4){1.5}}
\put(19.5,4){\line(3,-4){1.5}}
\put(19.1,1){$K_{3}$}

\end{picture}
\end{center}
\caption{ }
\label{fig1}
\end{figure}


\begin{figure}[h]
\begin{center}
\setlength{\unitlength}{0.7cm}
\begin{picture}(30,4)

\put(0,2){\circle*{.2}}
\put(1.5,4){\circle*{.2}}
\put(3,2){\circle*{.2}}
\thicklines
\put(0,2){\line(1,0){3}}
\put(0,2){\vector(1,0){1.8}}
\put(0,2){\line(3,4){1.5}}
\put(0,2){\vector(3,4){.9}}
\put(1.3,1){$\Gamma_{1}$}
\put(5,2){\circle*{.2}}
\put(6.5,4){\circle*{.2}}
\put(8,2){\circle*{.2}}
\thicklines
\put(5,2){\line(1,0){3}}
\put(8,2){\vector(-1,0){1.8}}
\put(5,2){\line(3,4){1.5}}
\put(5,2){\vector(3,4){.9}}
\put(6.3,1){$\Gamma_{2}$}
\put(3.7,0.2){(a)}
\put(12,2){\circle*{.2}}
\put(13.5,4){\circle*{.2}}
\put(15,2){\circle*{.2}}
\thicklines
\put(12,2){\line(1,0){3}}
\put(12,2){\vector(1,0){1.8}}
\put(12,2){\line(3,4){1.5}}
\put(13.5,4){\vector(-3,-4){.9}}
\put(15,2){\line(-3,4){1.5}}
\put(15,2){\vector(-3,4){.9}}
\put(13.1,1){$\Gamma_{3}$}
\put(17,2){\circle*{.2}}
\put(18.5,4){\circle*{.2}}
\put(20,2){\circle*{.2}}
\thicklines
\put(17,2){\line(1,0){3}}
\put(20,2){\vector(-1,0){1.8}}
\put(17,2){\line(3,4){1.5}}
\put(18.5,4){\vector(-3,-4){.9}}
\put(18.5,4){\line(3,-4){1.5}}
\put(20,2){\vector(-3,4){.9}}
\put(18.1,1){$\Gamma_{4}$}
\put(15.7,0.2){(b)}

\end{picture}
\end{center}
\caption{}\label{fig2}
\end{figure}

The proof of our main theorem is based on a result of Boussa\"{\i}ri et al
\cite{BILT} about the relationship between hemimorphy and clan decomposition
of digraphs. Proposition \ref{Biltdigraph} below is a special case of this result.

\begin{prop}
\label{Biltdigraph}Let $G=(V,E)$ be a finite simple graph and let $G^{\sigma}%
$, $G^{\tau}$ be two orientations of $G$. Then the following statements are equivalent:

\begin{description}
\item[i)] $G^{\sigma}[X]$ and $G^{\tau}[X]$ are hemimorphic, for any subset
$X$ of $V$ of size $3$;

\item[ii)] There exists a sequence $\sigma_{0}=\sigma,\ldots,\sigma_{m}=\tau$
of orientations of $G$ such that for $i=0,\ldots,m-1$, $G^{\sigma_{i+1}%
}=Inv(X_{i},G^{\sigma_{i}})$ where $X_{i}$ is a clan of $G^{\sigma_{i}}$.
\end{description}
\end{prop}


\section{Proof of Main theorem}
Let $G=(V,E)$ be a graph whose vertices are $v_{1},v_{2},\ldots,v_{n}$.
An\emph{\ }orientation of $G$ can be seen as a skew-symmetric map $\sigma$
from $V\times V$ to the set $\{0,1,-1\}$ such that $\sigma(i,j)=1$ if an only
if $(v_{i},v_{j})$ is an arc. Such orientation is denoted by $G^{\sigma}$.

Let $G^{\sigma}$ be an orientation of $G$. The \emph{skew-adjacency} matrix of
$G^{\sigma}$ is the real skew-symmetric matrix $S(G^{\sigma})=\left[
s_{i,j}\right]  $ where $s_{i,j}=1$ and $s_{j,i}=-1$ if $\left(  i,j\right)  $
is an arc of $G^{\sigma}$, otherwise $s_{i,j}=s_{j,i}=0$. Clearly, the entries
of $S(G^{\sigma})$ depend on the ordering of vertices. But the value of the
determinant $\det(S(G^{\sigma}))$ is independent of this ordering. So, we can
write $\det(G^{\sigma})$ instead of $\det(S(G^{\sigma}))$.

Consider now a skew-symmetric $\left \{  -1,0,1\right \}  $-matrix $A$. We
associate to $A$ its \emph{underlying graph} $G$ with vertex set $\left[
n\right]  $ and such that $\left \{  i,j\right \}  $ is a edge of $G_{\text{ }}%
$iff $a_{ij}\neq0$. Let $\sigma$ be the map from $\left[  n\right]
\times \left[  n\right]  $ to the set $\{0,1,-1\}$ such that $\sigma
(i,j)=a_{ij}$. Clearly, $G^{\sigma}$ is the unique orientation of $G$ such
that $S(G^{\sigma})=A$.

\begin{rem}
\label{correspond invers clan and HL clan} Let $G$ $=(\left[  n\right]  ,E)$
be a graph and let $G^{\sigma}$ be an orientation of $G$. Then:

\begin{description}
\item[i)] For every subset $X$ of $\left[  n\right]  $, we have
$S(Inv(X,G^{\sigma}))=Inv(X,S(G^{\sigma}))$;

\item[ii)] $Inv(\left[  n\right]  ,G^{\sigma})=(G^{\sigma})^{\ast}=G^{-\sigma
}$;

\item[iii)] Every clan of $G^{\sigma}$ is an HL-clan of $S(G^{\sigma})$.
\end{description}
\end{rem}

In addition to Corollary \ref{same pp}, the proof of our Main Theorem requires
the following Lemma.

\begin{lem}
\label{det4impliq3hem} Given a graph $G$ with four vertices $i,j,k,l$ such
that $i$ is adjacent to $j,k,l$. Let $G^{\sigma}$, $G^{\tau}$ be two
orientations of $G$. If $i\overset{G^{\sigma}}{\rightarrow}\left \{
j,k,l\right \}  $ , $i\overset{G^{\tau}}{\rightarrow}\left \{  j,k,l\right \}  $
and $\det(G^{\sigma})=\det(G^{\tau})$ then $G^{\sigma}\left[  j,k,l\right]  $
and $G^{\tau}\left[  j,k,l\right]  $ are hemimorphic.
\end{lem}

\begin{proof}
By remark \ref{fourcase}, we have four cases to consider.

\begin{description}
\item[i)] If $G\left[  j,k,l\right]  $ is the empty graph then $G^{\tau
}\left[  j,k,l\right]  =G^{\sigma}\left[  j,k,l\right]  $.

\item[ii)] If $G\left[  j,k,l\right]  $ is the graph $\overline{P_{2}}$ then
$G^{\tau}\left[  j,k,l\right]  =G^{\sigma}\left[  j,k,l\right]  $ or $G^{\tau
}\left[  j,k,l\right]  =(G^{\sigma}\left[  j,k,l\right]  )^{\ast}$.

\item[iii)] If $G\left[  j,k,l\right]  $ is the path $P_{2}$ and $G^{\sigma
}\left[  j,k,l\right]  $ is hemimorphic to $\Gamma_{1}$ then $\det(G^{\sigma
})=4$ and $G^{^{\tau}}\left[  j,k,l\right]  $ is hemimorphic to $\Gamma_{1}$
or $\Gamma_{2}$. The case when $G^{^{\tau}}\left[  j,k,l\right]  $ is
hemimorphic to $\Gamma_{2}$ implies that $\det(G^{\tau})=0$, which is
impossible. Analougsly, if $G^{\sigma}\left[  j,k,l\right]  $ is hemimorphic
to $\Gamma_{2}$ then $G^{^{\tau}}\left[  j,k,l\right]  $ must be hemimorphic
to $\Gamma_{2}$.

\item[iv)] If $G\left[  j,k,l\right]  $ is the complete graph $K_{3}$ and
$G^{\sigma}\left[  j,k,l\right]  $ is hemimorphic to $\Gamma_{3}$ then
$\det(G^{\sigma})=9$ and $G^{^{\tau}}\left[  j,k,l\right]  $ is hemimorphic to
$\Gamma_{3}$ or $\Gamma_{4}$. As in iii), the case when $G^{^{\tau}}\left[
j,k,l\right]  $ is hemimorphic to $\Gamma_{4}$ implies that $\det(G^{\tau}%
)=1$, which is impossible. Analougsly, if $G^{\sigma}\left[  j,k,l\right]  $
is hemimorphic to $\Gamma_{4}$ then $G^{^{\tau}}\left[  j,k,l\right]  $ must
be hemimorphic to $\Gamma_{4}$.
\end{description}
\end{proof}

\begin{pot}
 The implication ii)$\Longrightarrow$i) is obvious. To
prove iii)$\Longrightarrow$ii), it suffices to apply Corollary \ref{same pp}.
Let us prove that i) implies iii). As all off-diagonal entries of the first
row in $A$ and $B$ are non zeros, then there are two $\{-1,1\}$-diagonal
matrices $D$ and $D^{\prime}$ such that the first row of $A^{\prime}%
:=D^{-1}AD$ (resp $B^{\prime}:=D^{\prime-1}BD^{\prime}$) is $(0,1,1,\ldots
,1)$. By construction, $A^{\prime}$ and $B^{\prime}$ have the same underlying
graph $G$. Let $G^{\sigma}$ (resp. $G^{\tau}$) be the unique orientation of
$G$ such that $S(G^{\sigma})=A^{\prime}$ (resp. $S(G^{\tau})=B^{\prime}$). We
will show that i) of Proposition \ref{Biltdigraph} hold for $G^{\sigma}$ and
$G^{\tau}$. For this, let $X=\left \{  j,k,l\right \}  $ be a subset of $\left[
n\right]  $ of size $3$. If $1\in X$ (for example $j=1$), then $1\overset
{G^{\sigma}}{\rightarrow}\left \{  k,l\right \}  $, $1\overset{G^{\tau}%
}{\rightarrow}\left \{  k,l\right \}  $ and hence $G^{\sigma}\left[
j,k,l\right]  $ is isomorphic to $G^{\tau}\left[  j,k,l\right]  $. Assume now
that $1\notin X$ and let $Y:=\left \{  1,j,k,l\right \}  $. We have
$1\overset{G^{\sigma}}{\rightarrow}\left \{  j,k,l\right \}  $ and
$1\overset{G^{\tau}}{\rightarrow}\left \{  j,k,l\right \}  $. Moreover, as $A$
and $A^{\prime}$ (resp. $B$ and $B^{\prime}$) are digonally similar, we have
$\det(A^{\prime}\left[  Y\right]  )=\det(A\left[  Y\right]  )$, $\det
(B^{\prime}\left[  Y\right]  )=\det(B\left[  Y\right]  )$ and hence
$\det(A^{\prime}\left[  Y\right]  )=\det(B^{\prime}\left[  Y\right]  )$
because $A$ and $B$ have equal corresponding principal minors of order $4$.
Now, by definition, we have $\det(G^{\sigma}\left[  Y\right]  )=\det
(A^{\prime}\left[  Y\right]  )$ and $\det(G^{\tau}\left[  Y\right]
)=\det(B^{\prime}\left[  Y\right]  )$. It follows that $\det(G^{\sigma}\left[
Y\right]  )=\det(G^{\tau}\left[  Y\right]  )$ and then by Lemma
\ref{det4impliq3hem}, $G^{\sigma}\left[  j,k,l\right]  $ and $G^{\tau}\left[
j,k,l\right]  $ are hemimorphic. Now, from Proposition \ref{Biltdigraph},
there exists a sequence of orientations $\sigma_{0}=\sigma,\ldots,\sigma
_{m}=\tau$ of $G$ such that for $i=0,\ldots,m-1$, $G^{\sigma_{i+1}}%
=Inv(X_{i},G^{\sigma_{i}})$ where $X_{i}$ is a clan of $G^{\sigma_{i}}$. Let
$A_{i}^{\prime}:=S(G^{\sigma_{i}})$ for $i=0,\ldots,m$. By Remark
\ref{correspond invers clan and HL clan}, $X_{i}$ is is an HL-clan of
$A_{i}^{\prime}$ and $A_{i+1}^{\prime}=Inv(X_{i},A_{i}^{\prime})$ for
$i=0,\ldots,m-1$. We conclude by applying Proposition \ref{DSTimpliHLRE}%
.$\allowbreak$
\end{pot}


\end{document}